\documentclass{amsart}
\usepackage{amsmath}
\usepackage{amsfonts}
\usepackage{amssymb}
\usepackage{graphicx}%

\newtheorem{myrem}{Remark}[section]
\newtheorem{mytheo}{Theorem}[section]
\newtheorem{mylem}{Lemma}[section]
\newtheorem*{mypro}{Proof}

\newtheorem{mycor}{Corollary}[section]

\usepackage{color}

\theoremstyle{remark}

\numberwithin{equation}{section}

 \allowdisplaybreaks[4]
\begin{document}

\title{Uniform Hanson-Wright Type Deviation Inequalities for $\alpha$-Subexponential Random Vectors}


\author{Guozheng Dai}
\address{Zhejiang University, Hangzhou, 310027,  China.}
\email{11935022@zju.edu.cn}

\author{Zhonggen Su}
\address{Zhejiang University, Hangzhou, 310027,  China.}
\email{suzhonggen@zju.edu.cn}

\date{}

\keywords{$\alpha$-subexponential random variable, chaining argument, uniform Hanson-Wright inequality, restricted isometry property
}

\begin{abstract}
This paper is devoted to uniform versions of the Hanson-Wright inequality for a random vector with independent centered $\alpha$-subexponential entries, $0< \alpha\le 1$. Our method relies upon a novel decoupling inequality and a comparison of weak and strong moments. As an application, we use the derived inequality to prove the restricted isometry property of  partial random circulant matrices generated by standard $\alpha$-subexponential random vectors, $0<\alpha\le 1$.
\end{abstract}

\maketitle

\section{Introduction and Main Result}
Let $A=(a_{ij})_{n\times n}$ be an $n\times n$ fixed symmetric matrix and $S_{A}(\xi):=\sum_{i, j}a_{ij}\xi_{i}\xi_{j}$, where $\xi=(\xi_{1}, \cdots, \xi_{n})$ is a random vector with independent centered entries. A well-known concentration property of $S_{A}(\xi)$ is due to Hanson and Wright, claiming that if $\xi_{i}$ are independent centered subgaussian variables with $\Vert\xi_{i}\Vert_{\Psi_{2}}\le L$, then for all $t\ge 0$ (a modern version in \cite{Rudelson_ecp})
\begin{align}
	\textsf{P}\big\{\vert S_{A}(\xi)-\textsf{E}S_{A}(\xi)\vert \ge t  \big\}\le 2\exp\Big(-c\min\Big\{\frac{t^{2}}{L^{4}\Vert A\Vert_{F}^{2}}, \frac{t}{L^{2}\Vert A\Vert_{l_{2}\to l_{2}}}  \Big\}   \Big).\nonumber
\end{align}
Here, $\Vert \cdot\Vert_{F}$  and $\Vert \cdot\Vert_{l_{2}\to l_{2}}$ are the Frobenius norm and the spectral norm of a matrix (see \eqref{Notations} below), respectively. The Hanson-Wright type inequalities have found numerous applications in non-asymptotic theory of random matrices, statistics, and so on \cite{Bougain_gafa,Chen_X.H._Bernoulli,Koep_ACHA,Krahmer_CPAM,Zhou_S.H._Bernoulli}.

An important extension of such results is to consider the behaviour of random quadratic forms simultaneously for a family of $n\times n$ fixed symmetric matrices $\mathcal{A}$, that is the concentration property of $Z_{\mathcal{A}}(\xi):=\sup_{A\in\mathcal{A}}\sum_{i, j}a_{ij}\xi_{i}\xi_{j}$. When $\xi_{i}$ are independent Rademacher variables (take values $\pm 1$ with equal probabilities) and the diagnal entries of $A$ are $0$, Talagrand's celebrated work \cite{Talagrand_invention} yields for $t\ge 0$
\begin{align}\label{Eq_concentration_Talagrand}
	\textsf{P}\big\{\vert Z_{\mathcal{A}}(\xi)-\textsf{E}Z_{\mathcal{A}}(\xi)\vert \ge t  \big\}\le 2\exp\Big(-c\min\Big\{\frac{t^{2}}{(\textsf{E}\sup_{A\in \mathcal{A}} \Vert A\xi\Vert_{2})^{2}}, \frac{t}{\sup_{A\in\mathcal{A}}\Vert A\Vert_{l_{2}\to l_{2}}}  \Big\}   \Big).
\end{align}

Recently, Klochkov and Zhivotovskiy \cite{Klochkov_EJP} obtained a uniform Hanson-Wright type inequality for more general random variables via the entropy method. In particular, let $\xi_{i}$ be independent centered subgaussian variables and $L^{\prime}=\big\Vert\max_{i}\vert\xi_{i}\vert\big\Vert_{\Psi_{2}}$, then (without the condition $a_{ii}=0$)
\begin{align}\label{Eq_Klochkov}
	\textsf{P}\big\{ Z_{\mathcal{A}}(\xi)-\textsf{E}Z_{\mathcal{A}}(\xi)\ge t   \big\}
	\le \Big( -c\min\Big\{ \frac{t^{2}}{L^{\prime 2}(\textsf{E}\sup_{A\in \mathcal{A}}\Vert A\xi\Vert_{2})^{2}}, \frac{t}{L^{\prime 2}\sup_{A\in \mathcal{A}}\Vert A\Vert_{l_{2}\to l_{2}}}  \Big\}  \Big),
\end{align}
where $t\ge \max\{ L^{\prime}\textsf{E}\sup_{A\in\mathcal{A}}\Vert A\xi\Vert, L^{\prime 2}\sup_{A\in\mathcal{A}}\Vert A\Vert_{l_{2}\to l_{2}} \}$.

We usually call the results of the form \eqref{Eq_concentration_Talagrand} and \eqref{Eq_Klochkov} (one-sided) concentration inequalities. Another similar bounds are one-sided and have a multiplicative constant before $\textsf{E}Z_{\mathcal{A}}(\xi)$ (or replace $\textsf{E}Z_{\mathcal{A}}(\xi)$ with a suitable upper bound). This type of bound is called a deviation inequality. We next introduce some known uniform Hanson-Wright type deviation inequalities.

We first introduce some notations. Denote $\big(\textsf{E}\vert \xi_{1}\vert^{p}\big)^{1/p}$ by $\Vert \xi_{1}\Vert_{L_{p}}$ for a random variable $\xi_{1}$. Set $\alpha^{*}=\alpha/(\alpha-1)$ as the conjuate exponent of $\alpha$. Define $M_{F}(\mathcal{A}):=\sup_{A\in\mathcal{A}}\Vert A\Vert_{F}$ and $M_{l_{p_{1}}\to l_{p_{2}}}(\mathcal{A}):=\sup_{A\in\mathcal{A}}\Vert A\Vert_{l_{p_{1}}\to l_{p_{2}}}$ (see \eqref{Notations} below for the definition of the norm). We also denote Talagrand's $\gamma_{\alpha}$-functional with respect to $(\mathcal{A}, \Vert\cdot\Vert_{l_{2}\to l_{\alpha^{*}}})$ by $\gamma_{\alpha}(\mathcal{A}, \Vert\cdot\Vert_{l_{2}\to l_{\alpha^{*}}})$ (see \eqref{Eq_gamma_functional} below). Let
\begin{align}
	\Gamma(\alpha, \beta, \mathcal{A})=\gamma_{2}(\mathcal{A}, \Vert\cdot\Vert_{l_{2}\to l_{2}})+\gamma_{\alpha}(\mathcal{A}, \Vert\cdot\Vert_{l_{2}\to l_{\beta}})\nonumber
\end{align}
and 
\begin{align}
	U_{1}(\alpha, \beta)=& \Gamma(\alpha, \beta,  \mathcal{A}) \big(\Gamma(\alpha, \beta,  \mathcal{A})+M_{F}(\mathcal{A}) \big),\nonumber\\
	U_{2}(\alpha, \beta)=&M_{l_{2}\to l_{2}}(\mathcal{A})\Gamma(\alpha, \beta, \mathcal{A})+\sup_{A\in\mathcal{A}}\Vert A^\top A\Vert_{F},\nonumber\\
	U_{3}(\alpha, \beta)=&M_{l_{2}\to l_{\beta}}(\mathcal{A})\Gamma(\alpha, \beta, \mathcal{A}) .\nonumber
\end{align}

A widely used deviation inequality in the field of compressive sensing \cite{Bougain_gafa,Koep_ACHA,Krahmer_CPAM} was obtained by Krahmer et al \cite{Krahmer_CPAM}. Let $\xi_{i}$ be independent centered subgaussian variables and  $L=\max_{i}\Vert\xi\Vert_{\Psi_{2}}$. They proved for $t\ge 0$
\begin{align}\label{Eq_Krahmer}
	&\textsf{P}\Big\{ \sup_{A\in\mathcal{A}}\big\vert\Vert A\xi\Vert_{2}^{2}-\textsf{E}\Vert A\xi\Vert_{2}^{2}  \big\vert\ge CL^{2} \big(U_{1}(2, 2)+t\big) \Big\}
	\le 2\exp\Big(-\min\Big\{\frac{t^{2}}{U_{2}^{2}(2, 2)}, \frac{t}{M^{2}_{l_{2}\to l_{2}}(\mathcal{A})}\Big\}  \Big),
\end{align}  
where $\mathcal{A}$ is a family of $m\times n$ fixed matrices. 

Dai et al \cite{Dai_logconcave_RIP} extended \eqref{Eq_Krahmer} to the case where $\xi_{i}$ are independent centered $\alpha$-subexponential variables, $1\le \alpha\le 2$. Recall that a random variable $\xi_{1}$ is $\alpha$-subexponential (or $\alpha$-sub-Weibull) if satisfying
\begin{align}
	\textsf{P}\{\vert \xi_{1}\vert\ge Kt  \}\le ce^{-ct^{\alpha}},\quad t\ge 0,\nonumber
\end{align}
where $K$ is a parameter and $c$ is a universal constant.
The $\alpha$-subexponential (quasi-)norm of $\xi_{1}$ is defined as follows:
\begin{align}
	\Vert \xi_{1}\Vert_{\Psi_{\alpha}}:=\inf\{ t>0:\textsf{E}\exp(\frac{\vert \xi_{1}\vert^{\alpha}}{t^{\alpha}})\le 2\}.\nonumber
\end{align}
There are many interesting choices of random variables of this type, such as subexponential variables ($\alpha=1$), subgaussian variables ($\alpha=2$), or bounded variables (for any $\alpha>0$). 

Let $\xi_{i}$ be independent centered $\alpha$-subexponential variables and $L=L(\alpha)=\max_{i}\Vert\xi_{i}\Vert_{\Psi_{\alpha}}$. Dai et al (see Corollary $1.2^{*}$ in \cite{Dai_logconcave_RIP}) proved for $t\ge 0$
\begin{align}\label{Eq_Dai}
	&\textsf{P}\Big\{ \sup_{A\in \mathcal{A}}\big\vert \Vert A\xi\Vert_{2}^{2}-\textsf{E}\Vert A\xi\Vert_{2}^{2}  \big\vert>C(\alpha)L^{2}(U_{1}(\alpha, \alpha^{*})+t)\Big\}\nonumber\\
	\le& C_{1}(\alpha)\exp\Big(-\min\Big\{(\frac{t}{U_{2}(\alpha, \alpha^{*})})^{2},( \frac{t}{U_{3}(\alpha, \alpha^{*})})^{\alpha}, (\frac{t}{M_{l_{2}\to l_{2}}^{2}(\mathcal{A})})^{\alpha/2}\Big\}\Big),
\end{align}
where $\mathcal{A}$ is a family of $m\times n$ fixed matrices.

Our main result shows a uniform Hanson-Wright type deviation inequality for a random vector with independent centered $\alpha$-subexponential variables ($0<\alpha\le 1$), reading as follows:
\begin{mytheo}\label{Theo_main}
	Let $\xi=(\xi_{1}, \cdots, \xi_{n})$ be a random vector with independent centered $\alpha$-subexponential entries ($0<\alpha\le 1$) and $\mathcal{A}$ be a family of fixed $m\times n$ matrices. Then, we have for $t\ge 0$
	\begin{align}
		&\textsf{P}\Big\{ \sup_{A\in \mathcal{A}}\big\vert \Vert A\xi\Vert_{2}^{2}-\textsf{E}\Vert A\xi\Vert_{2}^{2}  \big\vert>C(\alpha)L^{2}(U_{1}(\alpha, \infty)+t)\Big\}\nonumber\\
		\le& C_{1}(\alpha)\exp\Big(-\min\Big\{(\frac{t}{U_{2}(\alpha, \infty)})^{2},( \frac{t}{U_{3}(\alpha, \infty)})^{\alpha}, (\frac{t}{M_{l_{2}\to l_{2}}^{2}(\mathcal{A})})^{\alpha/2}\Big\}\Big),\nonumber
	\end{align}
	where $L=L(\alpha)=\max_{i}\Vert \xi_{i}\Vert_{\Psi_{\alpha}}$.
\end{mytheo}
\begin{myrem}
	(i) In the setting of Theorem \ref{Theo_main}, G\"{o}tze et al showed for $t\ge 0$ (see Proposition 1.1 in \cite{Gotze_EJP})
	\begin{align}
		\textsf{P}\Big\{ \big\vert \Vert A\xi\Vert_{2}^{2}-\textsf{E}\Vert A\xi\Vert_{2}^{2}  \big\vert>C(\alpha)L^{2}t\Big\}
		\le C_{1}(\alpha)\exp\Big(-\min\Big\{(\frac{t}{\Vert A^{\top}A\Vert_{F}})^{2}, (\frac{t}{\Vert A^{\top}A\Vert_{l_{2}\to l_{2}}})^{\alpha/2}\Big\}\Big).\nonumber
	\end{align}
	Theorem \ref{Theo_main} recovers this result when $\mathcal{A}$ contains only one matrix.
	
	(ii) When $\alpha=1$, Theroem \ref{Theo_main} yields the same result as \eqref{Eq_Dai}.
\end{myrem}

In the case $\alpha =2$, the uniform Hanson-Wright deviation inequality (see \eqref{Eq_Krahmer}) is derived through the application of the classic decoupling inequality (see Theorem 2.5 in \cite{Krahmer_CPAM}) and the majorizing measure theorem (an optimal bound for the suprema of the stochastic process, e.g., (2.2) in \cite{Krahmer_CPAM}). To extend the deviation inequality to the case $1\le \alpha\le 2$, Dai et al \cite{Dai_logconcave_RIP} establish a novel decoupling inequality (see Lemma \ref{Lem_decoupling} below) and then obtain the desired result via the corresponding majorizing measure theorem (see Lemma 2.7 in \cite{Dai_logconcave_RIP}). For the case $0<\alpha\le 1$, although the novel decoupling inequality is valid (see Remark \ref{Rem_2} below), the majorizing measure theorem is absent, presenting a significant challenge to obtaining the deviation inequalities.

The main contribution of this paper is to show a uniform Hanson-Wright type deviation inequality for $\alpha$-subexponential ($0<\alpha\le 1$) entries. Our method of proof is based on a combination of the novel decoupling inequality with a comparison of weak and strong moments (see Lemma \ref{Lem_comparison}), the latter of which plays the same role as the majorizing measure theorem.

The remainder of this paper is structured as follows: Section 2 will present some notations and auxiliary lemmas. In Section 3, we shall prove our main result, Theorem \ref{Theo_main}. Finally, we will give an application of our main result in Section 4. In particular, we show the R.I.P. for partial random circulant matrices generated by standard $\alpha$-subexponential  ($0<\alpha\le 1$) random vectors.

\section{Preliminaries}
\subsection{Notations}
For a fixed vector $x=(x_{1},\cdots, x_{n})^\top\in \mathbb{R}^{n}$, we denote by $\Vert x\Vert_{p}=(\sum\vert x_{i}\vert^{p})^{1/p}$ the $l_{p}$ norm. Set $S^{n-1}_{p}=\{x\in\mathbb{R}^{n}: \Vert x\Vert_{p}=1 \}$ and $B^{n}_{p}=\{x\in\mathbb{R}^{n}: \Vert x\Vert_{p}\le 1 \}$.
We use $\Vert \xi\Vert_{L_p}=(\textsf{E}\vert \xi_{1}\vert^{p})^{1/p}$ for the $L_{p}$ norm of a random variable $\xi_{1}$.
As for an $m\times n$ matrix $A=(a_{ij})$, recall the following notations of norms
\begin{align}\label{Notations}
	\Vert A\Vert_{F}=&\sqrt{\sum_{i,j}\vert a_{ij}\vert^{2}}, \quad \Vert A\Vert_{\infty}=\max_{ij}\vert a_{ij}\vert, \quad \Vert A\Vert_{l_{p}(l_{2})}=(\sum_{i\le m}(\sum_{j\le n}\vert a_{ij}\vert^{2})^{p/2})^{1/p}\nonumber\\
	&\Vert A\Vert_{l_{p_{1}}\to l_{p_{2}}}=\sup \{\vert\sum a_{ij}x_{j}y_{i}\vert: \Vert x\Vert_{p_{1}}\le 1, \Vert y\Vert_{p_{2}^{*}} \le 1  \},
\end{align}
where $p_{2}^{*}=p_{2}/(p_{2}-1)$. We remark that $\Vert A\Vert_{l_{2}\to l_{2}}$ is the important spectral norm of $A$ and $\Vert A\Vert_{l_{1}\to l_{\infty}}=\Vert A\Vert_{\infty}$.

Unless otherwise stated, we denote by $C, C_{1}, c, c_{1},\cdots$ universal constants, and by $C(\delta), c(\delta),\cdots $ constants that depend only on the parameter $\delta$. 
For convenience, we write $f\lesssim g$ if $f\le Cg$ for some universal constant $C$ and write $f\lesssim_{\delta} g$ if $f\le C(\delta)g$ for some constant $C(\delta)$. We also say $f\asymp g$ if $f\lesssim g$ and $g\lesssim f$, so does $f\asymp_{\delta} g$. 

We say $\xi$ is  a symmetric Weibull variable with the scale parameter $1$ and the shape parameter $\alpha$ if $-\log \textsf{P}\{\vert\xi\vert>x  \}=x^{\alpha}, x\ge 0$. For convenience, we shall write $\xi\sim \mathcal{W}_{s}(\alpha)$. Given a nonempty set $T\subset l_{2}:=\{t: \sum_{i}t_{i}^{2}<\infty \}$,  we call $\{S_{t}: t\in T \}$ a canonical process, where $S_{t}=\sum_{i=1}^{\infty}t_{i}\xi_{i}$ and $\{\xi_{i}\}$ are independent random variables.

For a metric space $(T, d)$, a sequence of subsets $\{T_{r}: r\ge 0 \}$ of $T$ is admissible if for every $r\ge 1, \vert T_{r}\vert\le 2^{2^{r}}$ and $\vert T_{0}\vert=1$. For any $0<\alpha<\infty$, the definition of Talagrand's $\gamma_{\alpha}$-functional of $(T, d)$ is  as follows:
\begin{align}\label{Eq_gamma_functional}
	\gamma_{\alpha}(T, d)=\inf \sup_{t\in T}\sum_{r\ge 0}2^{r/\alpha}d(t, T_{r}),
\end{align}
where the infimum is taken concerning all admissible sequences of $T$. 

Given a set $T$, a sequence of partitions $(\mathcal{T}_{n})_{n\ge 0}$ of $T$ is admissible if satisfies $\vert \mathcal{T}_{0}\vert=1$, $\vert \mathcal{T}_{n}\vert\le 2^{2^{n}}$ for $n\ge 1$ and every set of $\mathcal{T}_{n+1}$ is contained in a set of $\mathcal{T}_{n}$ (increasing partitions). Denote  the unique element of $\mathcal{T}_{n}$ which contains $t$ by $T_{n}(t)$ and  the diameter of the set $T$ by $\Delta_{d}(T)$. Let
\begin{align}
	\gamma_{\alpha}^{\prime}(T, d)=\inf_{\mathcal{T}}\sup_{t\in T}\sum_{n\ge 0}2^{n/\alpha}\Delta_{d}(T_{n}(t)),\nonumber
\end{align}
where the infimum is taken over all admissible partitions $\mathcal{T}$ of $T$. Talagrand \cite{Talagrand_annals_prob} showed $\gamma_{\alpha}(T, d)\le \gamma^{\prime}_{\alpha}(T, d)\le C(\alpha)\gamma_{\alpha}(T, d)$.

\subsection{Tails and Moments}
In this subsection, we shall present some properties about the tails  and the  moments of random variables. 
\begin{mylem}[Lemma 4.6 in \cite{Latala_Inventions}]\label{Lem_moments1}
	Let $\xi$ be a random variable such that
	\begin{align}
		\theta_{1}p^{1/\alpha}\le \Vert \xi\Vert_{L_{p}}\le \theta_{2}p^{1/\alpha} \quad \text{for all}\,\, p\ge 2.\nonumber
	\end{align}
	Then there exist constants $c_{1}, c_{2}$ depending only on $\theta_{1}, \theta_{2}$ and $\alpha$ such that
	\begin{align}
		c_{1}e^{-t^{\alpha}/c_{1}}\le \textsf{P}\{\vert\xi\vert\ge t  \}\le c_{2}e^{-t^{\alpha}/c_{2}}\quad \text{for all}\,\, t\ge 0.\nonumber
	\end{align}
\end{mylem}
\begin{mylem}[Lemma 2.1 in \cite{Dai_logconcave_RIP}]\label{Lem_Moments_2}
	Assume that a random variable $\xi$ satisfies for $p\ge p_{0}$
	\begin{align}
		\Vert \xi\Vert_{L_p}\le \sum_{k=1}^{m}C_{k}p^{\beta_{k}}+C_{m+1},\nonumber
	\end{align}
	where $C_{1},\cdots, C_{m+1}> 0$ and $\beta_{1},\cdots, \beta_{m}>0$. Then we have for any $t>0$,
	\begin{align}
		\textsf{P}\big\{ \vert \xi\vert>e(mt+C_{m+1}) \big\}\le e^{p_{0}}\exp\Big(-\min\big\{\big(\frac{t}{C_{1}}\big)^{1/\beta_{1}},\cdots, \big(\frac{t}{C_{m}}\big)^{1/\beta_{m}}\big\}\Big)\nonumber
	\end{align}
	and
	\begin{align}
		\textsf{P}\Big\{\vert \xi\vert>e \big(\sum_{k=1}^{m}C_{k}t^{\beta_{k}}+C_{m+1}\big)  \Big\}\le e^{p_{0}}e^{-t}.\nonumber
	\end{align}
\end{mylem}

\begin{mylem}[Example 3 in \cite{Latala_SPL}]\label{Lem_alpha_0}
	Let $\xi_{1}, \cdots,\xi_{n}\stackrel{i.i.d.}{\sim}\mathcal{W}_{s}(\alpha)$.	Assume $0<\alpha\le 1$, then we have for $p\ge 2$
	
	(i) \begin{align}
		\big\Vert \sum_{i\le n}a_{i}\xi_{i}\big\Vert_{L_{p}}\asymp_{\alpha} p^{1/2}\Vert a\Vert_{2}+p^{1/\alpha}\Vert a\Vert_{\infty} ,\nonumber
	\end{align}
	where $a=(a_{1},\cdots, a_{n})^\top$ is a fixed vector. 
	
	(ii)\begin{align}\label{Eq_Lem_optimalbound_1}
		\big\Vert \sum_{i, j}a_{ij}\xi_{i}\tilde{\xi}_{j}\big\Vert_{L_{p}}\asymp_{\alpha}p^{1/2}\Vert A\Vert_{F}+p\Vert A\Vert_{l_{2}\to l_{2}}
		+p^{(\alpha+2)/2\alpha}\Vert A\Vert_{l_{2}\to l_{\infty}}+p^{2/\alpha}\Vert A\Vert_{\infty},
	\end{align}
	where $A=(a_{ij})$ is a fixed symmetric matrix and $\{\tilde{\xi}_{i}, i\le n\}$ are independent copies of $\{\xi_{i}, i\le n\}$.
\end{mylem}

\begin{myrem}\label{Rem_explanations}
	Observe that 
	\begin{align}
		\Vert A\Vert_{\infty}\le \Vert A\Vert_{l_{2}\to l_{\infty}}\le \Vert A\Vert_{l_{2}\to l_{2}}.\nonumber
	\end{align}
	Hence, \eqref{Eq_Lem_optimalbound_1} yields the following simplified bound
	\begin{align}
		\big\Vert \sum_{i, j}a_{ij}\xi_{i}\tilde{\xi}_{j}\big\Vert_{L_{p}}\lesssim_{\alpha} p^{1/2}\Vert A\Vert_{F}+p^{2/\alpha}\Vert A\Vert_{l_{2}\to l_{2}}.\nonumber
	\end{align}
\end{myrem}

\subsection{Contraction Principle}
In this subsection, we shall present a well-known contraction principle. 

\begin{mylem}[Lemma 4.6 in \cite{Ledoux_Talagrand_book}.]\label{Lem_contraction_principle}
	Let $F: \mathbb{R}^{+}\to \mathbb{R}^{+}$ be a convex function. Let further $\{\eta_{i}, i\le n\}$ and $\{\xi_{i}, i\le n\}$ be two symmetric sequences of independent random variables such that for some constant $K\ge 1$ and all $i\le n$ and $t>0$
	\begin{align}
		\textsf{P}\{\vert \eta_{i}\vert>t  \}\le K\textsf{P}\{\vert\xi_{i}\vert >t \}.\nonumber
	\end{align}
	Then, for any finite sequence $\{a_{i}, i\le n\}$ in a Banach space,
	\begin{align}
		\textsf{E}F\Big( \big\Vert \sum_{i=1}^{n}\eta_{i}a_{i}\big\Vert    \Big)\le \textsf{E}F\Big(K \big\Vert \sum_{i=1}^{n}\xi_{i}a_{i}\big\Vert   \Big).\nonumber
	\end{align}
	
\end{mylem}
\begin{mycor}\label{Cor_contraction_principle}
	Let $T$ be a nonempty subset of $\mathbb{R}^{n}$. $\{\eta_{i}\}$ and $\{\xi_{i}\}$ are independent random variables as in Lemma \ref{Lem_contraction_principle}. Then, we have for $p\ge 1$
	\begin{align}
		\textsf{E}\sup_{t\in T}\big\vert \sum_{i=1}^{n}\eta_{i}t_{i} \big\vert^{p}\le K^{p} \textsf{E} \sup_{t\in T}\big\vert \sum_{i=1}^{n}\xi_{i}t_{i} \big\vert^{p}.\nonumber
	\end{align}
\end{mycor}
\begin{proof}
	We first assume that $\text{span}(T)=\mathbb{R}^{n}$. Then, define the following norm on $\mathbb{R}^{n}$:
	\begin{align}
		\Vert x\Vert=\sup_{t\in T}\big\vert \sum_{i=1}^{n}x_{i}t_{i}\big\vert,\quad x\in\mathbb{R}^{n}.\nonumber
	\end{align}
	Let $e_{i}, i\ge 1$ be vectors in $\mathbb{R}^{n}$ such that the $i$-th coordinate is $1$ and the other coordinates are $0$. Then Lemma \ref{Lem_contraction_principle} yields that
	\begin{align}
		\textsf{E}F\Big( \big\Vert \sum_{i=1}^{n}\eta_{i}e_{i}\big\Vert    \Big)\le \textsf{E}F\Big(K \big\Vert \sum_{i=1}^{n}\xi_{i}e_{i}\big\Vert   \Big).\nonumber
	\end{align}
	We obtain the desired result by taking $F(\cdot)=\vert\cdot\vert^{p}$.
	
	As for a general $T$, we let $T_{\delta}=T\cup \delta B_{2}^{n}$. Note that, $\text{span}(T_{\delta})=\mathbb{R}^{n}$. Hence, Corollary \ref{Cor_contraction_principle} is valid for $T_{\delta}$. Then, go with $\delta\to 0$ and obtain the desired result. 
\end{proof}

\subsection{Comparison of Weak and Strong Moments }\label{Section_majorizing_measur_theorem}

Let $\{S_{t}=\sum t_{i}\xi : t\in T \}$ be a canonical process. There is a trivial lower estimate:
\begin{align}
	\big\Vert \sup_{t\in T} \vert S_{t}\vert\big\Vert_{L_p}\ge \max\Big\{ \textsf{E}\sup_{t\in T}\vert S_{t}\vert , \sup_{t\in T}\Vert S_{t}\Vert_{L_p}    \Big\}.\nonumber
\end{align}
In some situations, the lower bound may be reversed:
\begin{align}\label{Eq_comparison_result}
	\big\Vert \sup_{t\in T} \vert S_{t}\vert\big\Vert_{L_p}\lesssim  \textsf{E}\sup_{t\in T}\vert S_{t}\vert + \sup_{t\in T}\Vert S_{t}\Vert_{L_p}.    
\end{align}
We call a result of the form \eqref{Eq_comparison_result} a comparison of weak and strong moments. In this subsection, we shall  introduce a comparison of weak and strong moments (maybe the most general version currently), which was proved by Latała and Strzelecka \cite{Latala_mathematika}.

\begin{mylem}[Theorem 1.1 in \cite{Latala_mathematika}]\label{Lem_comparison}
	Let $\xi_{1}, \cdots, \xi_{n}$ be independent centered random variables with finite moments satisfying
	\begin{align}
		\Vert \xi_{i}\Vert_{L_{2p}}\le \alpha\Vert \xi_{i}\Vert_{L_p}, \quad \forall p\ge 2,\quad i=1,\cdots, n,\nonumber
	\end{align}
	where $\alpha$ is a finite positive constant. Then we have for any $p\ge 1$ and any non-empty set $T\subset \mathbb{R}^{n}$
	\begin{align}
		\big\Vert \sup_{t\in T} \vert \sum_{i=1}^{n}t_{i}\xi_{i}\vert\big\Vert_{L_p}\lesssim_{\alpha}  \textsf{E}\sup_{t\in T}\vert \sum_{i=1}^{n}t_{i}\xi_{i}\vert + \sup_{t\in T}\Vert \sum_{i=1}^{n}t_{i}\xi_{i}\Vert_{L_p}. \nonumber   
	\end{align}
\end{mylem}

\begin{myrem}
	Let $\xi_{1}\sim\mathcal{W}_{s}(\alpha), 0<\alpha\le 1$. A direct integration yields for $p\ge 1$
	\begin{align}
		\textsf{E}\vert\xi_{1}\vert^{p}=\frac{p}{\alpha}\Gamma(\frac{p}{\alpha}),\nonumber
	\end{align}
	where  $\Gamma(\cdot)$ is the Gamma function. Hence, $\xi_{1}$ satisfies the condition in Lemma \ref{Lem_comparison}.
\end{myrem}

\subsection{Decoupling Inequality}
Let $S_{A}=\sum_{i, j\le n}a_{ij}\xi_{i}\xi_{j}$, where $\xi_{i}$ are independent centered random variables and $A=(a_{ij})$ is a fixed matrix. The decoupling technique is applied to control $S_{A}$ by its decoupled version $\tilde{S}_{A}$ of the form: $\tilde{S}_{A}=\sum_{i, j\le n}a_{ij}\xi_{i}\tilde{\xi}_{j}$. We next introduce a decoupling inequality.

\begin{mylem}[Proposition 3.1 in \cite{Dai_logconcave_RIP}]\label{Lem_decoupling}
	Let $F$ be a convex function satisfying $F(x)=F(-x), \forall x\in \mathbb{R}$.  Assume $\xi=(\xi_{1}, \cdots, \xi_{n})^\top$ and $ \eta=(\eta_{1}, \cdots, \eta_{n})^\top$ are random vectors with independent centered entries. Let for any $t>0$ and $i\ge 1$, the independent random variables $\xi_{i}, \eta_{i}$ satisfy for  some $c\ge 1$
	\begin{align}\label{Condition_tail_decoupling}
		\textsf{P}\{ \xi_{i}^{2}\ge t \}\le c\textsf{P}\{ c\vert \eta_{i}\tilde{\eta}_{i}\vert\ge t \},
	\end{align}
	where $\tilde{\eta}_{i}$ is an independent copy of $\eta_{i}$. 
	Then, there exists a constant $C$ depending only on $c$ such that
	\begin{align}
		\textsf{E}\sup_{A\in \mathcal{A}}F\big(\xi^\top A\xi-\textsf{E}\xi^\top A\xi\big)\le \textsf{E}\sup_{A\in \mathcal{A}}F\big( C\eta^\top A\tilde{\eta}\big)\nonumber
	\end{align}
	where  $\mathcal{A}$ is a set of $n\times n$ fixed matrices.
\end{mylem}
\begin{myrem}\label{Rem_2}
	Let $\xi_{i}$ be independent centered $\alpha$-subexponential variables and  $\eta_{1}, \cdots, \eta_{n}\stackrel{i.i.d.}{\sim}\mathcal{W}_{s}(\alpha)$. Note that for $i\le n$ and $t\ge 0$
	\begin{align}
		\textsf{P}\{\xi_{i}^{2}\ge t  \}\le ce^{-ct^{\alpha/2}}\le c\textsf{P}\{ \vert \eta_{i}\tilde{\eta}_{i}\vert \ge (\frac{c}{2})^{2/\alpha}t \}\le c_{1}\textsf{P}\{c_{1}\vert\eta_{i}\tilde{\eta}_{i}\vert \ge t \},\nonumber
	\end{align}
	where $c_{1}=\max\{c, (2/c)^{2/\alpha}  \}$ and $\tilde{\eta}_{i}$ is an independent copy of $\eta_{i}$. Hence, $\xi_{i}$ and $\eta_{i}$ mentioned above satisfy the condition in Lemma \ref{Lem_decoupling}.
\end{myrem}

\section{Proof for Theorem \ref{Theo_main}}

In this section, we shall prove our main result, a uniform Hanson-Wright type deviation inequality. For convenience of reading, we will divide the proof of the theorem into two parts. We first prove the following proposition, a $p$-th moment bound for the suprema of 
decoupled bilinear random forms.  

\begin{mypro}\label{Prop_pthmoments}
	Let $\zeta_{1},\cdots,  \zeta_{n}\stackrel{\text{i.i.d}}{\sim}\mathcal{W}_{s}(\alpha)$, $0<\alpha\le 1$. Then, we have for $p\ge 1$
	\begin{align}\label{Eq_Theo1.3_zong}
		\Big\Vert \sup_{A\in \mathcal{A}}\vert \zeta^\top A^\top A\tilde{\zeta}\vert\Big\Vert_{L_{p}}\lesssim_{\alpha}\sup_{A\in \mathcal{A}}\Vert \zeta^\top A^\top A\tilde{\zeta}\Vert_{L_{p}}+\Big\Vert\sup_{A\in\mathcal{A}} \Vert A\tilde{\zeta}\Vert_{2}\Big\Vert_{L_{p}}\big(\gamma_{2}(\mathcal{A}, \Vert\cdot\Vert_{l_{2}\to l_{2}})+\gamma_{\alpha}(\mathcal{A}, \Vert\cdot\Vert_{l_{2}\to l_{\infty}}) \big),
	\end{align}
	where $\zeta=(\zeta_{1},\cdots, \zeta_{n})^\top$ and $\tilde{\zeta}$ is an independent copy of $\zeta$. 
\end{mypro}

\begin{proof}
	For convienence, denote by $d_{2}$ and $d_{\infty}$ the distance on $\mathcal{A}$ induced by $\Vert\cdot\Vert_{l_{2}\to l_{2}}$ and $\Vert\cdot\Vert_{l_{2}\to l_{\infty}}$. By the definition of $\gamma^{\prime}_{\alpha}$ functionals, we can select two admissible sequences of partitions $\mathcal{T}^{(1)}=(\mathcal{T}^{(1)}_{n})_{n\ge 0}$ and $\mathcal{T}^{(2)}=(\mathcal{T}^{(2)}_{n})_{n\ge 0}$ of $\mathcal{A}$ such that
	\begin{align}
		\sup_{A\in \mathcal{A}}\sum_{n\ge 0}2^{n/2}\Delta_{d_{2}}(T_{n}^{(1)}(A))\le 2\gamma_{2}^{\prime}(\mathcal{A}, d_{2}),\quad \sup_{A\in \mathcal{A}}\sum_{n\ge 0}2^{n/\alpha}\Delta_{d_{\infty}}(T_{n}^{(2)}(A))\le 2\gamma_{\alpha}^{\prime}(\mathcal{A}, d_{\infty}).\nonumber
	\end{align}
	Let $\mathcal{T}_{0}=\{\mathcal{A}\}$ and 
	\begin{align}
		\mathcal{T}_{n}=\{T^{(1)}\cap T^{(2)}:  T^{(1)}\in \mathcal{T}^{(1)}_{n-1}, T^{(2)}\in \mathcal{T}^{(2)}_{n-1}   \},\quad n\ge 1\nonumber.
	\end{align}
	Then, $\mathcal{T}=(\mathcal{T}_{n})_{n\ge 0}$ is a sequence of  increasing partitions  and 
	\begin{align}
		\vert \mathcal{T}_{n}\vert\le \vert \mathcal{T}^{(1)}_{n-1}\vert\vert \mathcal{T}^{(2)}_{n-1}\vert\le 2^{2^{n-1}}2^{2^{n-1}}=2^{2^{n}}.\nonumber
	\end{align}
	We next define a subset $\mathcal{A}_{n}$ of $\mathcal{A}$ by selecting exactly one point from each $T\in \mathcal{T}_{n}$. By this means, we build an admissible sequence $(\mathcal{A}_{n})_{n\ge 0}$ of subsets of $\mathcal{A}$. Let $\pi=\{ \pi_{r}, r\ge 0\}$ be a sequence of functions $\pi_{n}: \mathcal{A}\to \mathcal{A}_{n}$ such that $\pi_{n}(A)=\mathcal{A}_{n}\cap T_{n}(A)$, where $T_{n}(A)$ is the element of $\mathcal{T}_{n}$ containing $A$. Let $l$ be the largest integer satisfying $2^{l}\le p$.
	
	We first make the following decomposition:
	\begin{align}\label{Eq_decomposition_1}
		\big\Vert \sup_{A\in \mathcal{A}}\vert \zeta^\top A^\top A\tilde{\zeta}\vert\big\Vert_{L_{p}}\le\Big\Vert \sup_{A\in \mathcal{A}}\big\vert \zeta^\top A^\top A\tilde{\zeta}-\zeta^\top \pi_{l}(A)^\top\pi_{l}(A)\tilde{\zeta}\big\vert\Big\Vert_{L_{p}}+\Big\Vert \sup_{A\in \mathcal{A}}\big\vert \zeta^\top \pi_{l}(A)^\top\pi_{l}(A)\tilde{\zeta}\big\vert\Big\Vert_{L_{p}}.
	\end{align}
	
	We have for any fixed $A\in\mathcal{A}$,
	\begin{align}\label{Eq_decomposition_2}
		\big\vert \zeta^\top A^\top A\tilde{\zeta}-\zeta^\top \pi_{l}(A)^\top\pi_{l}(A)\tilde{\zeta}\big\vert\le \sum_{r\ge l}\big\vert\zeta^\top \Lambda_{r+1}(A)^\top\pi_{r+1}(A)\tilde{\zeta}\big\vert+\sum_{r\ge l}\big\vert\zeta^\top \pi_{r}(A)^\top\Lambda_{r+1}(A)\tilde{\zeta}\big\vert,
	\end{align}
	where $\Lambda_{r+1}A=\pi_{r+1}(A)-\pi_{r}(A)$.
	
	Conditionally on $\tilde{\zeta}$, we have by Lemma \ref{Lem_alpha_0}
	\begin{align}
		\big\Vert \zeta^\top \Lambda_{r+1}(A)^\top\pi_{r+1}(A)\tilde{\zeta}\big\Vert_{L_{p}}\lesssim_{\alpha}&p^{1/2}\big\Vert\Lambda_{r+1}(A)^\top\pi_{r+1}(A)\tilde{\zeta} \big\Vert_{2}+p^{1/\alpha}\big\Vert\Lambda_{r+1}(A)^\top\pi_{r+1}(A)\tilde{\zeta} \big\Vert_{\infty}.\nonumber
	\end{align}
	Then, we have by Lemma \ref{Lem_Moments_2}
	\begin{align}
		\textsf{P}_{\zeta}\Big\{ \vert \zeta^\top S_{r+1}(A, \tilde{\zeta})\vert\ge C(\alpha)\big(\sqrt{t} \Vert S_{r+1}(A, \tilde{\zeta}) \Vert_{2}+t^{1/\alpha} \Vert S_{r+1}(A, \tilde{\zeta}) \Vert_{\infty}\big)\Big\}\le e^{2}e^{-t}.\nonumber
	\end{align}
	Here, $S_{r+1}(A, \tilde{\zeta})=\Lambda_{r+1}(A)^\top\pi_{r+1}(A)\tilde{\zeta}$. By the definition of the operator norms, we have 
	\begin{align}
		\Vert S_{r+1}(A, \tilde{\zeta}) \Vert_{2}\le \Vert\Lambda_{r+1}(A)\Vert_{l_{2}\to l_{2}}\sup_{A\in\mathcal{A}} \Vert A\tilde{\zeta}\Vert_{2}\nonumber
	\end{align}
	and 
	\begin{align}
		\Vert S_{r+1}(A, \tilde{\zeta}) \Vert_{\infty}\le \Vert\Lambda_{r+1}(A)\Vert_{l_{2}\to l_{\infty}}\sup_{A\in\mathcal{A}} \Vert A\tilde{\zeta}\Vert_{2}.\nonumber
	\end{align}
	Note that $\big\vert\{ \pi_{r}(A): A\in\mathcal{A} \}\big\vert\le \vert \mathcal{A}_{r}\vert\le 2^{2^{r}}$. Then, $$\big\vert\{ \Lambda_{r+1}(A)^\top\pi_{r+1}(A): A\in \mathcal{A} \}\big\vert\le 2^{2^{r+3}}.$$
	Denote  $\Omega_{t, p}$ by the event
	\begin{align}
		\bigcap_{r\ge l}\bigcap_{A\in\mathcal{A}}\Big\{ \vert\zeta^\top S_{r+1}(A, \tilde{\zeta}) \vert\le C(\alpha)\sup_{A\in\mathcal{A}} \Vert A\tilde{\zeta}\Vert_{2}\big(\sqrt{t}2^{\frac{r}{2}}\Vert\Lambda_{r+1}(A)\Vert_{l_{2}\to l_{2}}+t^{\frac{1}{\alpha}}2^{\frac{r}{\alpha}}\Vert\Lambda_{r+1}(A)\Vert_{l_{2}\to l_{\infty}}\big)\Big\}. \nonumber
	\end{align}
	For $t>16$, we have by  a union bound
	\begin{align}
		\textsf{P}_{\zeta}\{ \Omega_{t, p}^{c} \}\le \sum_{r\ge l}2^{2^{r+3}}e^{2}e^{-2^{r}t}\le C_{1}(\alpha)\exp\big( -c(\alpha)pt \big).\nonumber
	\end{align}
	Assuming the event $\Omega_{t, p}$ occurs, we have
	\begin{align}
		\sum_{r\ge l}\big\vert\zeta^\top S_{r+1}(A, \tilde{\zeta}) \big\vert
		\le C(\alpha)\sup_{A\in\mathcal{A}} \Vert A\tilde{\zeta}\Vert_{2}\Big(\sum_{r\ge l}\sqrt{t}2^{\frac{r}{2}}\Vert \Lambda_{r+1}A\Vert_{l_{2}\to l_{2}}+\sum_{r\ge l}t^{\frac{1}{\alpha}}2^{\frac{r}{\alpha}}\Vert\Lambda_{r+1}A\Vert_{l_{2}\to l_{\infty}}\Big).\nonumber
	\end{align}
	Note that we have $\pi_{r+1}(A), \pi_{r}(A)\in T_{r}(A)\subset T^{(1)}_{r-1}(A)$ and so
	\begin{align}
		\Vert \Lambda_{r+1}A\Vert_{l_{2}\to l_{2}}\le \Delta_{d_{2}}\big(T^{(1)}_{r-1}(A)\big).
	\end{align}
	Hence, we have by the definition  of $\mathcal{T}^{(1)}$
	\begin{align}
		\sum_{r\ge l}2^{\frac{r}{2}}\Vert \Lambda_{r+1}A\Vert_{l_{2}\to l_{2}}\le \sum_{r\ge l}2^{\frac{r}{2}}\Delta_{d_{2}}\big(T^{(1)}_{r-1}(A)\big)\le 2\sqrt{2}\gamma_{2}^{\prime}(\mathcal{A}, d_{2}).\nonumber
	\end{align}
	Analogously, we have
	\begin{align}
		\sum_{r\ge l}2^{\frac{r}{\alpha}}\Vert\Lambda_{r+1}A\Vert_{l_{2}\to l_{\alpha^{*}}}\le 2\cdot 2^{\frac{1}{\alpha}}\gamma_{\alpha}^{\prime}(\mathcal{A}, d_{\infty}).\nonumber
	\end{align}
	Thus,
	\begin{align}
		\sup_{A\in \mathcal{A}}\sum_{r\ge l}\vert\zeta^\top S_{r+1}(A, \tilde{\zeta}) \vert\le 4C(\alpha)\sup_{A\in\mathcal{A}}\Vert A\tilde{\zeta}\Vert_{2}\big( \sqrt{t}\gamma_{2}^{\prime}(\mathcal{A}, d_{2})+t^{\frac{1}{\alpha}}    \gamma_{\alpha}^{\prime}(\mathcal{A}, d_{\infty})\big).\nonumber
	\end{align}
	As a consequence, we have for $t>16$
	\begin{align}
		\textsf{P}\Big\{ \sup_{A\in \mathcal{A}}\sum_{r\ge l}\vert\zeta^\top S_{r+1}(A, \tilde{\zeta}) \vert> C_{2}(\alpha)\sup_{A\in\mathcal{A}}\Vert A\tilde{\zeta}\Vert_{2}t\big( \gamma_{2}^{\prime}(\mathcal{A}, d_{2})+    \gamma_{\alpha}^{\prime}(\mathcal{A}, d_{\infty})\big)   \Big\}
		\le C_{1}(\alpha)\exp(-c(\alpha)pt).\nonumber
	\end{align}
	A direct integration and taking expectation with respect to variables $\{ \tilde{\eta}_{i}, i\le n\}$ lead to
	\begin{align}
		\Big\Vert\sup_{A\in \mathcal{A}}\sum_{r\ge l}\vert\zeta^\top S_{r+1}(A, \tilde{\zeta}) \vert\Big\Vert_{L_{p}}\lesssim_{\alpha}\Big\Vert\sup_{A\in\mathcal{A}} \Vert A\tilde{\zeta}\Vert_{2}\Big\Vert_{L_{p}}\big( \gamma_{2}(\mathcal{A}, d_{2})+\gamma_{\alpha}(\mathcal{A}, d_{\infty})\big)\nonumber.
	\end{align}
	Following the same line, one can give a similar bound for 
	$\sup_{A\in \mathcal{A}}\sum_{r\ge l}\vert\zeta^\top\pi_{r}(A)^\top \Lambda_{r+1}(A)\tilde{\zeta} \vert.$
	By virtue of (\ref{Eq_decomposition_2}), we have for $p\ge 1$
	\begin{align}
		\Big\Vert \sup_{A\in \mathcal{A}}\vert \zeta^\top A^\top A\tilde{\zeta}-\eta^\top \pi_{l}(A)^\top\pi_{l}(A)\tilde{\zeta}\vert\Big\Vert_{L_{p}}
		\lesssim_{\alpha}\Big\Vert\sup_{A\in\mathcal{A}} \Vert A\tilde{\zeta}\Vert_{2}\Big\Vert_{L_{p}}\big( \gamma_{2}(\mathcal{A}, d_{2})+\gamma_{\alpha}(\mathcal{A}, d_{\infty})\big).\nonumber
	\end{align}
	
	As for the other part of (\ref{Eq_decomposition_1}), we have for $p\ge 1$
	\begin{align}
		\textsf{E}\sup_{A\in \mathcal{A}}\vert \zeta^\top \pi_{l}(A)^\top\pi_{l}(A)\tilde{\eta}\vert^{p}&\le \sum_{B\in T_{l}}	\textsf{E}\vert \zeta^\top B^\top B\tilde{\zeta}\vert^{p}\le e^{p}\sup_{A\in \mathcal{A}}\textsf{E}\vert \zeta^\top A^\top A\tilde{\zeta}\vert^{p}.\nonumber
	\end{align}
	Thus
	\begin{align}
		\Big\Vert\sup_{A\in \mathcal{A}}\vert \zeta^\top \pi_{l}(A)^\top\pi_{l}(A)\tilde{\zeta}\vert\Big\Vert_{L_{p}}\le e\sup_{A\in \mathcal{A}}\Vert \zeta^\top A^\top A\tilde{\zeta}\Vert_{L_{p}}.\nonumber
	\end{align}
	Then, we conclude the proof by virtue of (\ref{Eq_decomposition_1}). 
\end{proof}

\begin{proof}[Proof of Theorem \ref{Theo_main}]
	Let $\zeta=(\zeta_{1}, \cdots, \zeta_{n})$ and $\zeta_{1}, \cdots, \zeta_{n}\stackrel{i.i.d.}{\sim}\mathcal{W}_{s}(\alpha)$, $0< \alpha\le 1$. We have by Remark \ref{Rem_2} and Proposition \ref{Prop_pthmoments} 
	\begin{align}
		\Big\Vert\sup_{A\in\mathcal{A}} \big\vert \Vert A\xi\Vert_{2}^{2}-\textsf{E}\Vert A\xi\Vert_{2}^{2} \big\vert\Big\Vert_{L_{p}}\lesssim_{L, \alpha}\sup_{A\in \mathcal{A}}\Vert \zeta^\top A^\top A\tilde{\zeta}\Vert_{L_{p}}+\Big\Vert\sup_{A\in\mathcal{A}} \Vert A\tilde{\zeta}\Vert_{2}\Big\Vert_{L_{p}}\big(\gamma_{2}(\mathcal{A}, \Vert\cdot\Vert_{l_{2}\to l_{2}})+\gamma_{\alpha}(\mathcal{A}, \Vert\cdot\Vert_{l_{2}\to l_{\infty}}) \big)\nonumber
	\end{align}

	Set $S=\{ A^\top x: x\in B^{n}_{2}, A\in\mathcal{A} \}$.  Lemmas \ref{Lem_alpha_0} and \ref{Lem_comparison} yield that 
	\begin{align}\label{Eq_Lpbound}
		\Big\Vert\sup_{A\in\mathcal{A}} \Vert A\zeta\Vert_{2}\Big\Vert_{L_{p}}&=\Big(\textsf{E}\sup_{A\in\mathcal{A}, x\in B_{2}^{n}}\vert x^\top A\zeta\vert^{p}   \Big)^{1/p}=\Big(\textsf{E}\sup_{u\in S}\vert u^\top \zeta\vert^{p}   \Big)^{1/p}\lesssim_{\alpha} \textsf{E}\sup_{u\in S}\vert u^\top \zeta\vert+\sup_{u\in S}\Vert u^\top \zeta\Vert_{L_{p}}\nonumber\\
		&\lesssim_{\alpha} \textsf{E}\sup_{u\in S}\vert u^\top \zeta\vert+\sqrt{p}\sup_{u\in S}\Vert u\Vert_{2}+p^{1/\alpha}\sup_{u\in S}\Vert u\Vert_{\infty}\nonumber\\
		&=\textsf{E}\sup_{A\in \mathcal{A}}\Vert A\zeta\Vert_{2}+\sqrt{p}M_{l_{2}\to l_{2}}(\mathcal{A})+p^{1/\alpha}M_{l_{2}\to l_{\infty}}(\mathcal{A}).
	\end{align}
	We have by Lemma \ref{Lem_decoupling} and  Proposition \ref{Prop_pthmoments} 
	\begin{align}
		\textsf{E}\sup_{A\in \mathcal{A}}\Vert A\tilde{\zeta}\Vert_{2}^{2}&\le \textsf{E} \sup_{A\in\mathcal{A}}\big\vert \Vert A\tilde{\zeta}\Vert_{2}^{2}-\textsf{E}\Vert A\tilde{\zeta}\Vert_{2}^{2}\big\vert  +\sup_{A\in \mathcal{A}}\textsf{E}\Vert A\tilde{\zeta}\Vert_{2}^{2}	\lesssim_{\alpha}\textsf{E}\sup_{A\in \mathcal{A}}\vert \zeta^\top A^\top A\tilde{\zeta}\vert+M^{2}_{F}(\mathcal{A})\nonumber\\
		&\lesssim_{\alpha}\textsf{E}\sup_{A\in \mathcal{A}}\Vert A\tilde{\zeta}\Vert_{2}(\gamma_{2}(\mathcal{A}, \Vert\cdot\Vert_{l_{2}\to l_{2}})+\gamma_{\alpha}(\mathcal{A}, \Vert\cdot\Vert_{l_{2}\to l_{\infty}}))+M^{2}_{F}(\mathcal{A}).\nonumber
	\end{align}
	Therefore, 
	\begin{align}
		\textsf{E}\sup_{A\in \mathcal{A}}\Vert A\tilde{\zeta}\Vert_{2}\le \Big\Vert  \sup_{A\in \mathcal{A}}\Vert A\tilde{\zeta}\Vert_{2}\Big\Vert_{L_{2}}\lesssim_{\alpha}\gamma_{2}(\mathcal{A}, \Vert\cdot\Vert_{l_{2}\to l_{2}})+\gamma_{\alpha}(\mathcal{A}, \Vert\cdot\Vert_{l_{2}\to l_{\infty}})+M_{F}(\mathcal{A}).\nonumber
	\end{align}
	Then, we conclude the proof by Remark \ref{Rem_explanations} and Lemma \ref{Lem_Moments_2}.
\end{proof}

\section{Applications}
An $m\times n$ matrix $A$ has the restricted isometry property (R.I.P.) with parameters $\delta$ and $s$ if satisfying for all $s$-sparse $x$ (i.e., $\vert \{l: x_{l}\neq 0  \}\vert\le s$)
\begin{align}\label{Eqisometry}
	(1-\delta)\Vert x\Vert_{2}^{2}\le \Vert A x\Vert_{2}^{2}\le (1+\delta)\Vert x\Vert_{2}^{2}.
\end{align}
The restricted isometry constant $\delta_{s}$ is defined as the smallest number satisfying (\ref{Eqisometry}). In this section, we shall prove the R.I.P. of partial random circulant matrices (see the definition below).

Recall the  cyclic subtraction: $j\ominus k=(j-k)\mod n$. Denote  by $z*x=((z*x)_{1}, \cdots, (z*x)_{n})^\top$ the circular convolution of two vectors $x, y\in \mathbb{R}^{n}$, where
\begin{align}
	(z*x)_{j}:=\sum_{k=1}^{n}z_{j\ominus k}x_{k}.\nonumber
\end{align}
The circulant matrix $H=H_{z}$ generated by $z$ is a $n\times n$ matrix with entries $H_{jk}=z_{j\ominus k}$. Equivalently, $Hx=z*x$ for every $x\in \mathbb{R}^{n}$. 

Let $\Omega\subset \{1,\cdots, n\}$ be a fixed set with $\vert \Omega\vert=m$. Denote by $R_{\Omega}: \mathbb{R}^{n}\to \mathbb{R}^{m}$  the operator restricting a vector $x\in \mathbb{R}^{n}$ to its entries in $\Omega$. The partial circulant matrix generated by $z$ is defined as follows
\begin{align}
	\Phi=\frac{1}{\sqrt{m}}R_{\Omega}H_{z}.\nonumber
\end{align}
We say $\Phi$ a partial random circulant matrix when $z$ is a random vector. 

Krahmer et al \cite{Krahmer_CPAM} showed the R.I.P. of partitial random circulant matrices generated by standard subgaussian random vectors (i.e., random vectors with independent standard subgaussian entries). Dai et al \cite{Dai_logconcave_RIP} extended this result to the case where the generating random vectors are standard $\alpha$-subexponential vectors, $1\le \alpha\le 2$. The following result extends this result to the case $0<\alpha\le 1$.
\begin{mytheo}\label{Theo_RIP}
	Let $\Phi$ be an $m\times n$ partial random circulant matrix generated by a standard $\alpha$-subexponential random vector $\eta$, $0< \alpha\le 1$. Set $L=\max_{i\le n}\Vert \eta_{i}\Vert_{\Psi_{\alpha}}$. Then, under the condition $m\ge c_{1}(\alpha, L)\delta^{-2}s^{2/\alpha}\log^{4/\alpha} n,$ the restricted isometry constant $\delta_{s}$ satisfies
	\begin{align}
		\textsf{P}\{ \delta_{s}\le \delta \}\ge 1-\exp(-c_{0}(\alpha, L)s^{2/\alpha}\log^{4/\alpha} n).       \nonumber
	\end{align}
\end{mytheo}

\begin{proof}[Proof of Theorem \ref{Theo_RIP}]
	
	Let $\Omega$ be a subset of $\{1,\cdots, n\}$ with $\vert \Omega\vert=m$ and $P_{\Omega}=R_{\Omega}^\top R_{\Omega}$, where $R_{\Omega}: \mathbb{R}^{n}\to \mathbb{R}^{m}$ is the operator restricting a vector $x\in \mathbb{R}^{n}$ to its entries in $\Omega$.  In particular, $P_{\Omega}: \mathbb{R}^{n}\to \mathbb{R}^{n}$, is a projection operator such that
	\begin{align}
		(P_{\Omega}x)_{l}=x_{l},\,\, l\in \Omega;\quad (P_{\Omega}x)_{l}=0,\,\, l\notin \Omega.\nonumber
	\end{align}
	Let $D_{s, n}=\{ x\in \mathbb{R}^{n}: \Vert x\Vert_{2}\le 1, \Vert x\Vert_{0}\le s \}$. We have by (\ref{Eqisometry}), 
	\begin{align}\label{Eq_RIP_chaos}
		\delta_{s}&=\sup_{x\in D_{s, n}}\Big\vert \frac{1}{m}\Vert \Phi x\Vert_{2}^{2}-\Vert x\Vert_{2}^{2}\Big\vert=\sup_{x\in D_{s, n}}\Big\vert \frac{1}{m}\Vert P_{\Omega}(\eta *x)\Vert_{2}^{2}-\Vert x\Vert_{2}^{2}\Big\vert\nonumber\\
		&=\sup_{x\in D_{s, n}}\Big\vert \frac{1}{m}\Vert P_{\Omega}(x*\eta)\Vert_{2}^{2}-\Vert x\Vert_{2}^{2}\Big\vert=\sup_{x\in D_{s, n}}\Big\vert\Vert V_{x}\eta\Vert_{2}^{2}-\Vert x\Vert_{2}^{2}\Big\vert	\nonumber\\
		&=\sup_{x\in D_{s, n}}\Big\vert\Vert V_{x}\eta\Vert_{2}^{2}-\textsf{E}\Vert V_{x}\eta\Vert_{2}^{2}\Big\vert,
	\end{align}
	where the last equality is due to $\textsf{E}\Vert V_{x}\eta\Vert_{2}^{2}=\Vert x\Vert_{2}^{2}$.

	Let $\mathcal{A}=\{ V_{x}: x\in D_{s, n}\}$. To bound $\delta_{s}$, we  need to control the quantities $M_{F}(\mathcal{A}), M_{l_{2}\to l_{\beta}}(\mathcal{A})$ and $ \gamma_{\alpha}(\mathcal{A}, \Vert\cdot\Vert_{l_{2}\to l_{\beta}})$ ($\beta=2, \infty$) by virtue of Theorem \ref{Theo_main}. 
	
	Observe that the matrices $V_{x}$ have shifted copies of $x$ in all their $m$ nonzero rows. Thus, the $l_{2}$-norm of each nonzero row is $m^{-1/2}\Vert x\Vert_{2}$. Then, for $x\in D_{s, n}$, we have $\Vert V_{x}\Vert_{F}=\Vert x\Vert_{2}\le 1$, and thus $M_{F}(\mathcal{A})=1$.
	
	Note that
	\begin{align}
		\Vert V_{x}\Vert_{l_{2}\to l_{\beta}}\le \Vert V_{x}\Vert_{l_{2}\to l_{2}}\le \sup_{\Vert y\Vert_{2},\Vert z\Vert_{2}\le 1}\frac{1}{\sqrt{m}}\vert y^\top V_{x}z\vert\le \frac{1}{\sqrt{m}}\Vert x\Vert_{1}\le \sqrt{\frac{s}{m}}\Vert x\Vert_{2}\le \sqrt{\frac{s}{m}},\nonumber
	\end{align}
	It follows immediately  from that $M_{l_{2}\to l_{\beta}}(\mathcal{A})\le \sqrt{s/m}$ ($\beta=2, \infty$).
	
	Lastly, we use the covering number to bound $ \gamma_{\alpha}(\mathcal{A}, \Vert\cdot\Vert_{l_{2}\to l_{\beta}})$. In particular, we have
	\begin{align}
		\gamma_{\alpha}(\mathcal{A}, \Vert\cdot\Vert_{l_{2}\to l_{\beta}})\lesssim_{\alpha} \int_{0}^{\infty}\big( \log N(T, \Vert \cdot\Vert_{l_{2}\to l_{\beta}}), u  \big)^{1/\alpha} du,\nonumber
	\end{align} 
	where $N(T, d, u)$ is the covering number of $T$ with the distance $d$ and the radius $u$.
	Note that, Krahmer et al \cite{Krahmer_CPAM} showed the following bounds for $u\ge 1/\sqrt{m}$
	\begin{align}
		\log N(\mathcal{A}, \Vert\cdot\Vert_{l_{2}\to l_{2}}, u)\lesssim\frac{s}{m}(\frac{\log n}{u})^{2}\nonumber
	\end{align}
	and for $u\le 1/\sqrt{m}$
	\begin{align}
		\log N(\mathcal{A}, \Vert\cdot\Vert_{l_{2}\to l_{2}}, u)\lesssim s\log(\frac{en}{su}).\nonumber
	\end{align}
	A direct integration (Section 5.2 in \cite{Dai_logconcave_RIP} contains a detailed calculation) yields that for $0<\alpha\le 1$
	\begin{align}
		\gamma_{\alpha}(\mathcal{A}, \Vert\cdot\Vert_{l_{2}\to l_{\infty}})\le\gamma_{\alpha}(\mathcal{A}, \Vert\cdot\Vert_{l_{2}\to l_{2}})\lesssim_{\alpha}\frac{s^{1/\alpha}}{\sqrt{m}}\log^{2/\alpha} n\nonumber
	\end{align}
	and for $\alpha=2$
	\begin{align}
		\gamma_{2}(\mathcal{A}, \Vert\cdot\Vert_{l_{2}\to l_{2}})\lesssim\sqrt{\frac{s}{m}}\log s\log n.\nonumber 
	\end{align}

	Assume for any $0<\delta, \alpha\le 1$ 
	\begin{align}
		m\ge c_{1}(\alpha, L)\delta^{-2}\max\big\{(s^{2/\alpha}\log^{4/\alpha} n),     (s\log^{2}s\log^{2}n)\big\}=c_{1}(\alpha, L)\delta^{-2}s^{2/\alpha}\log^{4/\alpha} n.\nonumber
	\end{align}
	Then, we have
	\begin{align}
		\gamma_{2}(\mathcal{A}, \Vert\cdot\Vert_{l_{2}\to l_{2}}), \gamma_{\alpha}(\mathcal{A}, \Vert\cdot\Vert_{l_{2}\to l_{\infty}})\lesssim_{\alpha, L}\delta.\nonumber
	\end{align}
	Let $c_{1}(\alpha, L)$ be an appropriately constant such that  $C(\alpha)L^{2}U_{1}(\alpha, \infty)\le\delta/2$, where $C(\alpha)$ is the constant in Theorem \ref{Theo_main}. Then, we have
	\begin{align}
		\textsf{P}\{ \delta_{s}\ge \delta \}&\le \textsf{P}\{ \delta_{s}\ge C(\alpha)L^{2}U_{1}(\alpha, \infty)+\delta/2 \}\le C_{1}(\alpha)\exp(-c_{2}(\alpha, L)(m/s)^{\alpha/2}\delta^{\alpha}),\nonumber
	\end{align}
	which concludes the proof.
\end{proof}

\textbf{Acknowledgment} The work was partly supported by the National Natural Science Foundation of China (12271475, U23A2064).


\begin{thebibliography}{9}
	
	
	\bibitem{Bougain_gafa}
	Bourgain, J., Dirksen, S., Nelson, J. (2015).
	Toward a unified theory of sparse dimensionality reduction in Euclidean space.
	\textit{Geom. Funct. Anal.}
	\textbf{25(4)} 1009--1088.
	
	\bibitem{Chen_X.H._Bernoulli}
	Chen, X. H., Yang, Y. (2021).
	Hanson–Wright inequality in Hilbert spaces with application to K-means clustering
	for non-Euclidean data.
	\textit{Bernoulli}
	\textsf{27(1)} 586--614.
	
	\bibitem{Dai_logconcave_RIP}
	Dai, G.Z., Su, Z.G., Ulyanov, V., Wang, H.C. (2024).
	On log-concave-tailed chaos and the restricted isometry property.
	arXiv: 2401.14860.
	
	
	
	\bibitem{Gotze_EJP}
	G\"{o}tze, F., Sambale, H., Sinulis, A.(2021)
	Concentration inequalities for polynomials in $\alpha$-sub-exponential random variables.
	\textit{Electron. J. Probab.}
	\textbf{48(26)} 1--22.
	
	
	
	
	\bibitem{Klochkov_EJP}
	Klochkov, Y., Zhivotovskiy, N.(2020)
	Uniform Hanson-Wright type concentration inequalities for unbounded entries via the entropy method.
	\textit{Electron. J. Probab.}
	\textbf{25(22)} 1--30.
	
	
	
	\bibitem{Koep_ACHA}
	Koep, N., Behboodi, A., Mathar, R. (2022).
	The restricted isometry property of block diagonal matrices for group-sparse signal recovery.
	\textit{Appl. Comput. Harmon. Anal.}
	\textsf{60} 333--367.
	
	\bibitem{Latala_SPL}
	Kolesko, K., Latala, R.(2015).
	Moment estimates for chaos generated by symmetric random variables with logarithmically convex tails.
	\textit{Statist. Probab. Lett.}
	\textbf{107} 210--214.
	
	
	\bibitem{Krahmer_CPAM}
	Krahmer, F., Mendelson, S., Rauhut, H.(2014).
	Suprema of chaos processes and the restricted isometry property.
	\textit{Comm. Pure Appl. Math.}
	\textbf{67} 1877--1904.
	
	\bibitem{Latala_mathematika}
	Latała, R., Strzelecka, M.(2018).
	Comparison of weak and strong moments for vectors with independent coordinates.
	\textit{Mathematika}
	\textbf{64(1)} 211--229.	
	
	
	\bibitem{Latala_Inventions}
	Latała, R., van Handel, R. and Youssef, P.(2018).
	The dimension-free structure of nonhomogeneous random matrices.
	\textit{Invent. math.}
	\textbf{214} 1031--1080.
	
	\bibitem{Ledoux_Talagrand_book}
	Ledoux, M., Talagrand, M.(1991).
	Probability in Banach spaces. Isoperimetry and Processes.
	Ergebnisse der Mathematik und ihrer Grenzgebiete (3), 23.
	Springer, Berlin.
	
	
	\bibitem{Rudelson_ecp}
	Rudelson, M., Vershynin, R. (2013).
	Hanson-Wright inequality and sub-gaussian concentration.
	\textit{Electron. Commun. Probab.}
	\textbf{18} 82--91.	
	
	\bibitem{Talagrand_invention}
	Talagrand, M.(1996).
	New concentration inequalities in product spaces.
	\textit{Invent. math.}
	\textbf{126} 505--563.	
	
	\bibitem{Talagrand_annals_prob}
	Talagrand, M.(2001).
	Majorizing measures without measures
	\textit{Ann. Prob.}
	\textbf{29(1)} 411--417.
	
	
	\bibitem{Zhou_S.H._Bernoulli}
	Zhou, S.H. (2019).
	Sparse Hanson–Wright inequalities for subgaussian quadratic forms.
	\textit{Bernoulli}
	\textsf{25(3)} 1603--1639.
	
	
	
\end{thebibliography}
\end{document}